\documentclass[12pt,a4paper]{amsart}
\usepackage[english]{babel}
\usepackage{lmodern}
\usepackage{newlfont}
\usepackage{booktabs}
\usepackage{color}
\usepackage[T1]{fontenc}
\usepackage[utf8]{inputenc}
\usepackage{indentfirst}
\usepackage{amsmath}
\usepackage{amsmath,amssymb}
\usepackage{amsthm}
\usepackage{geometry}
\usepackage{mathtools}
\usepackage{listings}
\usepackage{amsfonts}
\usepackage{braket}
\usepackage{emptypage}
\usepackage{newlfont}
\usepackage{amssymb}
\usepackage{graphicx}
\usepackage[italian]{varioref}
\usepackage{accents}
\usepackage{comment}
\usepackage{stmaryrd}
\usepackage{pgfplots}

\makeatletter
\labelformat{equation}{\tagform@{#1}}
\makeatother
\usepackage{hyperref}
\usepackage{relsize}
\usepackage{faktor}
\usepackage{enumerate}
\usepackage{fancyhdr}

\usepackage[colorinlistoftodos]{todonotes}

\DeclarePairedDelimiter{\abs}{\lvert}{\rvert}

\DeclareMathOperator{\dist}{dist}

\renewcommand{\phi}{\varphi}
\renewcommand{\epsilon}{\varepsilon}

\newcommand{\numberset}{\mathbb}
\newcommand{\eps}{\varepsilon}
\newcommand{\e}{\varepsilon}
\newcommand{\N}{\numberset{N}}
\newcommand{\C}{\numberset{C}}
\newcommand{\R}{\numberset{R}}

\newcommand{\norm}[1]{\left\|#1\right\|}
\newcommand{\wto}{\rightharpoonup}

\newcommand{\curv}{\mathfrak{K}}
\newcommand{\tangvet}{\textsl{t}}
\newcommand{\orig}{\mathbf{0}}

\newcommand{\la}{\lambda}

\newcommand{\x}{\xi}
\newcommand{\y}{\eta}

\theoremstyle{definition}
\newtheorem{definition}{Definition}[section]
\theoremstyle{definition}                                                                         
\newtheorem{definizione}{Definizione}[section]
\theoremstyle{definition}                                                                         
\newtheorem{rmk}[definizione]{Remark}
\theoremstyle{plain}                                                                              
\newtheorem{thm}[definizione]{Theorem}
\theoremstyle{plain}    
\newtheorem{prop}[definizione]{Proposition}
\theoremstyle{plain}     

\theoremstyle{plain}

\theoremstyle{definition}

\numberwithin{equation}{section}

\begin{document}
\title[On the shape of solutions to elliptic equations]{On the shape of solutions to elliptic equations in possibly non convex planar domains}
\thanks{This work was supported by INDAM-GNAMPA}

\author{Luca Battaglia, Fabio De Regibus and Massimo Grossi}
\address[Luca Battaglia]{Dipartimento di Matematica e Fisica, Universit\`a degli Studi Roma Tre, Largo S. Leonardo Murialdo 1 - 00146 Roma, Italy. \newline e-mail: {\sf lbattaglia@mat.uniroma3.it}.}
\address[Fabio De Regibus]{Dipartimento di Matematica, Universit\`a degli Studi di Torino, Via Carlo Alberto 10 - 10123 Torino, Italy. \newline e-mail: {\sf fabio.deregibus@unito.it}.}
\address[Massimo Grossi]{Dipartimento di Scienze di Base e Applicate per l'Ingegneria, Universit\`a di Roma ``La Sapienza'', Via Antonio Scarpa 14 - 00161 Roma, Italy. \newline e-mail: {\sf massimo.grossi@uniroma1.it}.}

\begin{abstract}
In this note we prove uniqueness of the critical point for positive solutions of elliptic problems in bounded planar domains: we first examine the Poisson problem $-\Delta u=f(x,y)$ finding a geometric condition involving the curvature of the boundary and the normal derivative of $f$ on the boundary to ensure uniqueness of the critical point. In the second part we consider stable solutions of the nonlinear problem $-\Delta u =f(u)$ in perturbation of convex domains.
\end{abstract}

\maketitle

\section{Introduction and main results}

In this note we are interested in investigating the number of critical points of solutions of elliptic equations. Let $\Omega\subseteq\R^2$ be a smooth and bounded domain and consider the following problem
\[
\begin{cases}
-\Delta u=f\left((x,y),u\right)&\text{in }\Omega\\
u>0&\text{in }\Omega\\
u=0&\text{on }\partial\Omega,
\end{cases}
\]
where $f:\Omega\times\R\to\R$ is a smooth function.
It is known that the number of critical points of solutions of the previous problem is strongly influenced by the function $f$ and by the geometry of the domain $\Omega$. In this context the literature is too wide to report all the result, so we just focus on some of them which are closest to the interest of this paper (see~\cite{Ma16,gr21} for a recent survey on the topic in the case $f=f(u)$ and the monograph~\cite{KawBook}).

The first fundamental result has been proved for the solution of the torsion problem, i.e. $f\equiv1$ by Makar-Limanov~\cite{ML71}: uniqueness of critical point holds if $\Omega$ is convex. The same is true for the first eigenfunction, $f(u)=\lambda u$, see~\cite{bl,app}.

Under symmetry assumptions the well known results of Gidas, Ni and Nirenberg~\cite{gnn} gives the uniqueness of the critical point for general $f$.

Uniqueness of critical point can also be proved for solutions of the Schrödinger eigenvalue problem
\[
-\Delta u+V(x)u=\lambda u,
\]
with $V$ convex, $\lambda>0$. It follows from the results in~\cite{Ko83,CS82}, still in convex domains (their result holds also for more general equations).

If we consider the Poisson problem $-\Delta u= f(x,y)$, as a consequence of a result by Kennington in~\cite{Ken85} one can prove uniqueness of the critical point if $\Omega$ is convex and if $f$ is $\beta$-concave, i.e. $f^\beta$ is concave, for some $\beta>1$.

The last result we want to recall is due to Cabré and Chanillo~\cite{cc}. They show that if the curvature of $\partial\Omega$ is striclty positive and $u$ is a semi-stable solution of the nonlinear problem $-\Delta u=f(u)$, then $u$ has exactly one critical point (the result has been extended to the case of non negative curvature in~\cite{dgm}). Remember that $u$ is said to be a \emph{(semi-)stable} solution if the linearized operator at $u$ is (nonnegative) positive definite, i.e. if for all $\phi\in\mathcal C^{\infty}_{0}(\Omega)$ one has
\[
\int_{\Omega}|\nabla\phi|^{2}-\int_{\Omega} f'(u)|\phi|^{2}>(\ge)0,
\]
or equivalently if the first eigenvalue of the linearized operator $-\Delta-f'(u)$ in $\Omega$ is (nonnegative) positive.

All the results mentioned before hold in convex domain and it is known that, in general, we cannot expect uniqueness of the critical point in non convex domains, as showed by the case of the torsion problem in a dumbell domain (see for instance \cite{sperb}). Then it is natural to ask whether it is possible to recover the uniqueness in (possibly) non convex domain, under suitable assumptions.

In the first part of this paper we examine the Poisson problem
\begin{equation}
\label{PbPoisson}
\begin{cases}
-\Delta u=f(x,y)&\text{in }\Omega\\
u=0&\text{on }\partial\Omega,
\end{cases}
\end{equation}
where $f:\overline\Omega\to\R$ is a smooth function and $\Omega\subseteq\R^2$ is a smooth, bounded (and simply connected) domain. We have the following result.
\begin{thm}
\label{mainthm}
Assume $f>0$ in $\overline\Omega$ and
\begin{equation}
\label{hp2}
\Delta(\log f)=0,\quad\text{in }\Omega,
\end{equation}
and
\begin{equation}
\label{hp3}
\frac{1}{2}\frac{\partial f}{\partial\nu}+\curv f\ge0,\quad\text{on }\partial\Omega,
\end{equation}
where $\nu$ is the outnormal unit vector to $\partial\Omega$ and $\curv$ is its curvature.
If $u$ is the solution of problem~\ref{PbPoisson}, then it has a unique nondegenerate critical point $(x_0,y_0)\in\Omega$.
\end{thm}

\begin{rmk}
\begin{enumerate}
\item Theorem~\ref{mainthm} holds even if $\Omega$ is not convex. For instance, if we consider $f(x,y):=e^{2x}$ then equation~\ref{hp2} is trivially satisfied while equation~\ref{hp3} is satisfied if $\Omega$ is such that
\begin{equation}
\label{dom:exp}
\nu_x\ge-\curv,\quad\text{on }\partial\Omega,
\end{equation}
where we write $\nu:=(\nu_x,\nu_y)$. Hence, we can find a non convex domain $\Omega$ such that~\ref{dom:exp} is satisfied, see Figure~\ref{fig1}.
\item Let us point out that we cannot drop assumption~\ref{hp2} or~\ref{hp3}, otherwise we can lose the uniqueness of the critical point: see Remark~\ref{rmk1-contr} and Remark~\ref{rmk2-contr} for the details.
\item The preceding theorem can be seen as a generalization of Makar-Limanov's result in~\cite{ML71} for the torsion problem. Indeed, for $f\equiv1$ and $\Omega$ convex it is easy to see that assumptions~\ref{hp2} and~\ref{hp3} are trivially satisfied.
\end{enumerate}
\end{rmk}

\vspace{20pt}

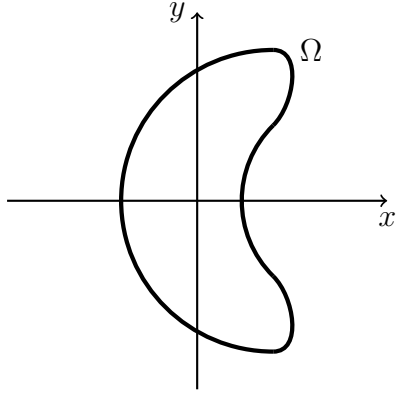
\begin{figure}[]
\begin{center}
	\begin{tikzpicture}
	
	\draw[thick,->] (-3.5,0) -- (1.5,0) node[below] {$x$};
	\draw[thick,->] (-1,-2.5) -- (-1,2.5) node[left] {$y$};

	\draw [ultra thick] (0,2) arc (90:270:2);
	\draw [ultra thick] (0,1) to [out=225,in=135] (0,-1);
	
	\draw [ultra thick] (0,2) to [out=0,in=45] (0,1);
	
	\draw [ultra thick] (0,-1) to [out=315,in=0] (0,-2);
	
	\node at (0.5,2) {$\Omega$};

	\end{tikzpicture}
	\end{center}
\caption{\label{fig1}Example of domain which satisfies~\ref{dom:exp}.}
\end{figure}

 The idea of the proof is the following: thanks to the assumptions~\ref{hp2} and~\ref{hp3} we can construct an invertible conformal map that maps $\Omega$ to a bounded and convex domain where we are able to estimate the number of critical points of the composition of our solution with the inverse of the conformal map.\\

In the second part of this paper, we want to investigate the nonlinear problem
\begin{equation}
\label{CAP4:PB0}
\begin{cases}
-\Delta u=f(u)&\hbox{in }\Omega\\
u>0&\hbox{in }\Omega\\
u=0&\hbox{on }\partial\Omega,
\end{cases}
\end{equation}
where $\Omega$ is still a smooth bounded and convex domain in $\R^2$ and $f:\R\to\R$ is a smooth nonlinearity.

We are interested in the number of critical points of solutions of the preceding problem when $\Omega$ is close to be a convex domain.

We recall that if $\Omega$ is convex, we know that a semi-stable solution $u$ of~\ref{CAP4:PB0} admits exactly one critical point which turns out to be nondegenerate, see~\cite{cc,dgm}. Then we can ask what happens if we consider domains which are (possibly) non convex, but close to a convex one.

First of all let us recall that for a convex domain $\Omega\subseteq\R^N$, with no empty interior, we can find a Lipschitz function $\chi_\Omega:\mathbb S^{N-1}\to(0+\infty)$ such that
\[
\Omega=\Set{tP(\left(1+\chi_\Omega(P)\right)|P\in\mathbb S^{N-1},\quad t\in[0,1)},
\]
where up to a translation and a dilatation we assumed without loss of generality that $B_1(\orig)\subseteq\Omega$. Furthermore, if we assume $\Omega$ to be of class $\mathcal C^k$ then $\chi_\Omega\in\mathcal C^k(\mathbb S^{N-1})$. 

Hence let us give the following definition of convergence of a family of smooth sets to a smooth and convex one.

\begin{definition}
\label{cap4:def}
Given a bounded and convex set $\Omega\subseteq \R^N$ of class $\mathcal C^k$ and with no empty interior, we say that the family $(\Omega_\eps)_\eps\subseteq\R^N$ of bounded sets of class $\mathcal C^k$ \emph{converges to the convex set} $\Omega$ for $\eps\to0$ - and we write $\Omega_\eps\to\Omega$ for $\eps\to0$ - if there exists a family of functions $(\chi_{\Omega_\eps})_\eps\subseteq\mathcal C^k(\mathbb S^{N-1})$ such that
\[
\Omega_\eps=\Set{tP(\left(1+\chi_{\Omega_\eps}(P)\right)|P\in\mathbb S^{N-1},\quad t\in[0,1)},
\]
and
\[
\norm{\chi_{\Omega_\eps}-\chi_{\Omega}}_{\mathcal C^k(\mathbb S^{N-1})}\to 0,\quad\text{as }\eps\to0.
\]
Here again we assumed, without loss of generality, that $B_1(\orig)\subseteq\Omega$.
\end{definition}

Now, let us fix a smooth and convex domain $\Omega\subseteq\R^2$, with $B_1(\orig)\subseteq\Omega$ and consider a family of domains $\Omega_\eps$ that are smooth and such that $\Omega_\eps\to\Omega$ for $\eps\to0$, say at least in $\mathcal C^4$ sense, according to the preceding definition.

The following result holds true.
\begin{thm}
\label{thm2}
Let $u_\eps$ be a semi-stable solution the following problem
\begin{equation}
\label{CAP4:PB1}
\begin{cases}
-\Delta u_\eps= f(u_\eps)&\text{ in }\ \Omega_\eps\\
u_\eps>0&\text{ in }\ \Omega_\eps\\
u_\eps=0&\text{ on }\ \partial\Omega_\eps,
\end{cases}
\end{equation}
with $f(0)\ge0$ and assume that $\norm{u_\eps}_{L^\infty(\Omega_\eps)}\le C$ for some $C>0$.
Then $u_\eps$ has a unique nondegenerate critical point $P_\eps\in\Omega_\eps$. Moreover $P_\eps\to P_c$ where $P_c\in\Omega$ is the unique critical point of a semi-stable solution $u$ of problem~\ref{CAP4:PB0} in $\Omega$.
\end{thm}

\begin{rmk}
As will be explained in details in Remark~\ref{rmk:hp:limitatezzanorma}, let us point out that the hypothesis that $\norm{u_\eps}_{L^\infty(\Omega_\eps)}$ is uniformly bounded is always satisfied if, for instance, we assume that the nonlinearity has the form $f(u)=\lambda g(u)$, $g$ is smooth and satisfies
\begin{gather}
\label{i2}g:\R\to\R\,\hbox{ is increasing and convex},\\
\label{i3}g(0)>0,
\end{gather}
and $\lambda\in(0,\lambda^*(\Omega))$, with $\lambda^*(\Omega)>0$. We recall that in this setting it is well known that there exists $\lambda^*(\Omega)>0$ such that for all $\lambda\in (0,\lambda^*(\Omega))$ the problem
\begin{equation}
\label{PB0}
\begin{cases}
-\Delta u=\lambda g(u)&\hbox{in }\Omega\\
u>0&\hbox{in }\Omega\\
u=0&\hbox{on }\partial\Omega
\end{cases}
\end{equation}
admits a  stable solution, see for instance~\cite{Bandle(book),cr75,mp} and the references therein.
\end{rmk}

\begin{rmk}
Observe that if the limit domain is convex, but unbounded, then the result of Theorem~\ref{thm2} does not hold. Indeed, it is possible to build a family of domains which converges to the strip $\mathcal S:=\R\times(-1,1)$ and such that the corresponding solutions have an arbitrary large (finite) number of critical points. More precisely for any $\lambda\in(0,\la^*(-1,1))$ and for all $k\in\N$ there exists a family of smooth and bounded domains $\widetilde\Omega_\e\subseteq\R^{2}$ such that
\begin{enumerate}[(i)]
\item $\widetilde\Omega_\e$ is star-shaped with respect to and interior point;
\item $\widetilde\Omega_\e$ locally converges to the strip $\mathcal S$ for $\e\to0$, i.e. for all compact sets $K\subseteq\R^{2}$ it holds $\abs{K\cap(\mathcal S\Delta\widetilde\Omega_\e)}\to0$ as $\e\to0$;
\item $\lambda^*(\widetilde\Omega_\e)\ge\la^*(-1,1)$ for $\e$ small enough;
\item if $\widetilde u_\e$ is the stable solution of problem~\ref{PB0} in $\widetilde\Omega_\e$ then $\widetilde u_\e$ has at least $k$ maximum points.
\end{enumerate}

This fact has been proved in~\cite{gg3} for the torsion problem $f\equiv1$ and generalized to any $f=\lambda g$, with $g$ satisfying~\ref{i2} and~\ref{i3} in~\cite{DRG}.\\
\end{rmk}

To prove Theorem~\ref{thm2} we show that $u_\eps$ converges to $u$, the solution of probem~\ref{CAP4:PB0} in $\Omega$, and then the claim can be deduced by the case of the convex domain.\\

The paper is organized as follows: in the next section we prove Theorem~\ref{mainthm}, and we conclude it by showing that if at least one between~\ref{hp2} and~\ref{hp3} does not hold, then Theorem~\ref{mainthm} may fail.
In Section~\ref{sez3} we prove Theorem~\ref{thm2}.
\newpage
\section{The Poisson problem}

In this section we prove Theorem~\ref{mainthm}.
Until the end of the section we identify $\R^2$ with $\C$ and we write $z:=x+iy$.

Assume $f:\overline\Omega\to\R$ be positive and such that assumptions~\ref{hp2} and~\ref{hp3} are satisfied. Then the following proposition holds true.
\begin{prop}
\label{prop1}
There exists a holomorphic function $T:\Omega\to\C$ such that
\begin{enumerate}[(i)]
\item $\abs{T'}^2=f$ in $\Omega$,
\item $T(\Omega)$ is bounded and convex,
\item there exists a holomorphic function $\tau:T(\Omega)\to\C$ such that $\tau=T^{-1}$.
\end{enumerate}
\end{prop}
\begin{proof}
Without loss of generality, we can assume $\orig\in\Omega$. 
Since $\log f$ is harmonic in $\Omega$ by~\ref{hp2} and $\Omega$ is simply connected, we can find a holomorphic function $w:\Omega\to\R$  such that
\[
\Re(w)=\frac{1}{2}\log f.
\]
Moreover, since $f>0$ in $\overline\Omega$, $w$ is holomorphic up to the boundary of $\Omega$.
Hence, also the function $\mathfrak t:=e^w$ is holomorphic in $\overline\Omega$ and if we decompose it by modulus and principal argument we have
\[
\mathfrak t=\abs{\mathfrak t}e^{i\Theta}.
\]
Finally
\[
\frac{1}{2}\log f=\Re(w)=\Re\left(\log\mathfrak t\right)=\Re\left(\log(\abs{\mathfrak t}e^{i\Theta})\right)=\log\abs{\mathfrak t},
\]
which yields to
\begin{equation}
\label{eq1}
\abs{\mathfrak t}^2=f.
\end{equation}
Then, since $\Omega$ is simply connected and $\mathfrak t$ is holomorphic, we can define $T$ as the primitive of $\mathfrak t$ such that $T(\orig)=\orig$. Clearly $T$ is holomorphic and~\ref{eq1} implies $\abs{T'}^2=f$, proving $(i)$.

To prove $(ii)$ note that since $T$ is continuous up to the boundary of $\Omega$ we have that $\abs{T}$ is bounded and then $T(\Omega)$ is.
To show the second claim we recall that the curvature $\widetilde\curv$ of the boundary of $T(\Omega)$, in $\zeta\in\partial T(\Omega)$ is given by (see~\cite[equation (23) pag. 234]{Ne97})
\begin{equation}
\label{eqcurv}
\widetilde\curv(\zeta)=\frac{1}{\abs{T'(z)}}\left(\Im\left(\frac{\tangvet(z)T''(z)}{T'(z)}\right)+\curv(z)\right),
\end{equation}
where $z\in\partial\Omega$ satisfies $T(z)=\zeta$ and $\tangvet(z)$ is the unit tangent vector to $\partial\Omega$ in $z$. Here, since $\Omega$ is simply connected then its boundary $\partial\Omega$ is a Jordan curve and we assume it is orientated is such a way that the winding number satisfies
\[
W_{\partial\Omega}(z):=\frac{1}{2\pi i}\int_{\partial\Omega}\frac{d\mathbf{z}}{\mathbf{z}-z}=\begin{cases} 1&\text{if }z\in\Omega\\0&\text{if }z\in\C\setminus\overline\Omega.\end{cases}
\]
Thanks to this convention the unit tangent vector $\tangvet$ is uniquely determined by the orientation of $\partial\Omega$.
 If we write $T:=h+ig$ we have
\[
T'=h_x-ih_y,\quad\text{and}\quad T''=h_{xx}-ih_{xy},
\]
and then, writing $\tangvet=x_\tangvet+iy_\tangvet$, one has
\begin{align}
\Im\left(\frac{\tangvet T''}{T'}\right)&=\Im\left(\frac{(x_\tangvet+iy_\tangvet)(h_{xx}-ih_{xy})}{h_x-ih_y}\right)\nonumber\\
			&=\Im\left(\frac{x_\tangvet h_{xx}+y_\tangvet h_{xy}+i(-x_\tangvet h_{xy}+y_\tangvet h_{xx})}{h_x-ih_y}\right)\nonumber\\
			&=\frac{x_\tangvet(h_yh_{xx}-h_xh_{xy})+y_\tangvet(h_yh_{xy}+h_xh_{xx})}{h_x^2+h_y^2}\nonumber\\
\label{eqalign}	&=\frac{y_\tangvet(h_xh_{xx}+h_yh_{xy})-x_\tangvet(h_xh_{xy}+h_yh_{yy})}{h_x^2+h_y^2}.
\end{align}
Taking into account
\begin{align*}
\partial_xf&=\partial_x(h_x^2+h_y^2)=2h_xh_{xx}+2h_yh_{xy},\\
\partial_yf&=\partial_y(h_x^2+h_y^2)=2h_xh_{xy}+2h_yh_{yy},
\end{align*}
equation~\ref{eqalign} becomes
\[
\Im\left(\frac{\tangvet T''}{T'}\right)=\frac{y_\tangvet f_x-x_\tangvet f_y}{2f}=\frac{\nu\cdot\nabla f}{2f},
\]
where 
$\nu=(y_\tangvet,-x_\tangvet)$. Finally, the previous equation and~\ref{eqcurv} imply
\[
\abs{T'}\widetilde\curv=\Im\left(\frac{\tangvet T''}{T'}\right)+\curv=\frac{f_\nu+2\curv f}{2f}\ge0,
\]
where the last inequality holds true by~\ref{hp3}. Then $\widetilde\curv\ge0$ and $T(\Omega)$ is convex.

Since $T$ is proper, $T'\not=0$ and $T(\Omega)$ is simply connected,~\cite[Theorem B]{Go72} tells us that $T$ is invertible. Finally, the inverse is holomorphic by the Open Mapping Theorem and $(iii)$ follows.
\end{proof}

\begin{rmk}
In particular $T$ is a conformal map, indeed $\abs{T'}^2=f>0$ in $\overline\Omega$ and then $\Omega$ and $T(\Omega)$ are conformally equivalent.
\end{rmk}

\begin{rmk}
The function $T$ can be written in a more explicit way by setting
\[
T'(z)=\mathfrak t(x+iy):=\frac{1}{\overline{\mathfrak t(\orig)}}f\left(\frac{x+iy}{2},\frac{y-ix}{2}\right).
\]
See~\cite[equation (4.3)]{Sh04}.
\end{rmk}

We can now prove  Theorem~\ref{mainthm}.

\begin{proof}[Proof of Theorem~\ref{mainthm}]
Let us denote $\Lambda:=T(\Omega)$ with coordinates $\zeta:=\x+i\y$ and set
\[
v(\x,\y):=u(\tau(\x,\y)),
\]
where we recall that $\tau=T^{-1}$. If we write $\tau:=\phi+i\psi$ we have
\begin{align*}
\partial_\x v&=u_x\phi_\x+u_y\psi_\x,\\
\partial_\y v&=u_x\phi_\y+u_y\psi_\y,
\end{align*}
and
\begin{align*}
\partial_{\x\x} v&=u_{xx}\phi_\x^2+2u_{xy}\phi_\x\psi_\x+u_{yy}\psi_\x^2+u_x\phi_{\x\x}+u_y\psi_{\x\x},\\
\partial_{\y\y} v&=u_{xx}\phi_\y^2+2u_{xy}\phi_\y\psi_\y+u_{yy}\psi_\y^2+u_x\phi_{\y\y}+u_y\psi_{\y\y}.
\end{align*}
Hence by the Cauchy-Riemann equations one has
\begin{align*}
\Delta v&=u_{xx}\abs{\nabla\phi}^2+2u_{xy}\nabla\phi\nabla\psi+u_{yy}\abs{\nabla\psi}^2+u_x\Delta\phi+u_y\Delta\psi=\Delta u\abs{\tau'}^2=-f\abs{\tau'}^2,
\end{align*}
and then by $(i)$ of Proposition~\ref{prop1} we get
\[
-\Delta v=f\abs{T'}^{-2}=1,
\]
that is $v$ is the solution of the torsion problem in $\Lambda$, i.e.
\[
\begin{cases}
-\Delta v=1&\text{in }\Lambda\\
v=0&\text{on }\partial\Lambda.
\end{cases}
\]
By~\cite[Theorem 1]{ML71}, $v$ has a unique nondegenerate critical point $(\x_0,\y_0)\in\Lambda$ and then $(x_0,y_0):=\tau(\x_0,\y_0)\in\Omega$ is the unique critical point of $u$.
To show the nondegeneracy of $(x_0,y_0)$, since $u_x(x_0,y_0)=u_y(x_0,y_0)=0$ one has
\begin{align*}
\partial_{\x\x} v(\tau(\x_0,\y_0))&=u_{xx}\phi_\x^2+2u_{xy}\phi_\x\psi_\x+u_{yy}\psi_\x^2,\\
\partial_{\x\y} v(\tau(\x_0,\y_0))&=u_{xx}\phi_\x\phi_\y+u_{xy}(\phi_\x\psi_\y+\phi_\y\psi_\x)+u_{yy}\psi_\x\psi_\y,\\
\partial_{\y\y} v(\tau(\x_0,\y_0))&=u_{xx}\phi_\y^2+2u_{xy}\phi_\y\psi_\y+u_{yy}\psi_\y^2.
\end{align*}
Then
\begin{align*}
\partial_{\x\x} v\partial_{\y\y} v&=u_{xx}^2\phi_\x^2\phi_\y^2+2u_{xx}u_{xy}\phi_\x^2\phi_\y\psi_\y+u_{xx}u_{yy}\phi_\x^2\psi_\y^2+\\
	&+2u_{xx}u_{xy}\phi_\x\phi_\y^2\psi_\x+4u_{xy}^2\phi_\x\phi_\y\psi_\x\psi_\y+2u_{xy}u_{yy}\phi_\x\psi_\x\psi_\y^2+\\
	&+u_{xx}u_{yy}\phi_\y^2\psi_\x^2+2u_{xy}u_{yy}\phi_\y\psi_\x^2\psi_\y+u_{yy}^2\psi_\x^2\psi_\y^2,
\end{align*}
and
\begin{align*}
\left(\partial_{\x\y} v\right)^2&=u_{xx}^2\phi_\x^2\phi_\y^2+u_{xy}^2(\phi_\x\psi_\y+\phi_\y\psi_\x)^2+u_{yy}^2\psi_\x^2\psi_\y^2+2u_{xx}u_{yy}\phi_\x\phi_\y\psi_\x\psi_\y+\\&
+2u_{xx}u_{xy}\phi_\x^2\phi_\y\psi_\y+2u_{xx}u_{xy}\phi_\x\phi_\y^2\psi_\x+2u_{xy}u_{yy}\phi_\x\psi_\x\psi_\y^2+2u_{xy}u_{yy}\phi_\y\psi_\x^2\psi_\y.
\end{align*}
Finally we have that
\[
v_{\x\x}v_{\y\y}-v_{\x\y}^2=(u_{xx}u_{yy}-u_{xy}^2)(\phi_\x\psi_\y-\phi_\y\psi_\x)^2,
\]
and since $(\x_0,\y_0)$ is nondegenerate, the same holds true for $(x_0,y_0)$.
\end{proof}

\subsection{Final remarks}

We conclude this section by showing that if at least one between~\ref{hp2} and~\ref{hp3} does not hold, then Theorem~\ref{mainthm} may fail.

\begin{rmk}
\label{rmk1-contr}
If $f>0$ in $\overline\Omega$, satisfies~\ref{hp2}, but does not satisfy~\ref{hp3} then the solution of the Poisson problem~\ref{PbPoisson} can have more than one critical point.

Indeed in~\cite{gg3} it is shown that for any $\delta>0$, there exists a star-shaped domain $\Omega:=\Omega(\delta)$ such that the solution of the torsion problem, i.e. $f\equiv 1$, admits at least two critical points. Moreover one has $\curv_{|\partial\Omega}\ge-\delta$ and it is negative somewhere. Then $f>0$ in $\overline\Omega$,~\ref{hp2} is satisfied but~\ref{hp3} is not.
\end{rmk}

\begin{rmk}
\label{rmk2-contr}
If $f>0$ in $\overline\Omega$, satisfies~\ref{hp3}, but does not satisfy~\ref{hp2} then the solution of the Poisson problem~\ref{PbPoisson} can have more than one critical point.

Indeed, as a consequence of~\cite[Theorem 1.1]{EPW}, one has that if $p>1$ is large enough there exists a solution $u$ of the following Hénon problem
\[
\begin{cases}
-\Delta u=(x^2+y^2)^\alpha \abs{u}^p &\text{in }B\\
u>0 &\text{in }B\\
u=0 &\text{on }\partial B,
\end{cases}
\]
with $\alpha>0$, $B:=B_1(\orig)$, and there exist $Q_1,Q_2\in B$ such that
\[
\max_{B\setminus \bigcup_{i=1}^2B_{2\delta}(Q_i)} u\le \frac{\sqrt e}{4},\quad\text{and}\quad\sup_{B_\delta(Q_i)} u\ge \frac{\sqrt e}{2},\quad i=1,2,
\]
for some $0<\delta<\frac{\dist(Q_1,Q_2)}{4}$. Then let $v$ be the solution of the torsion problem in $B$ with Dirichlet boundary conditions and set
\[
u_\eps:=u+\eps v,\quad 0<\eps<\frac{\sqrt e}{4\norm{v}_{\infty}},
\]
which solves
\[
\begin{cases}
-\Delta u_\eps=f_\eps(x,y) &\text{in }B\\
u_\eps>0 &\text{in }B\\
u_\eps=0 &\text{on }\partial B,
\end{cases}
\]
with $f_\eps(x,y):=(x^2+y^2)^\alpha \abs{u}^p+\eps>0$ in $\overline B$. Then $u=0$ on $\partial B$ implies
\[
\frac{1}{2}\frac{\partial f_\eps}{\partial\nu}+\curv f_\eps=0+\eps>0,\quad\text{on }\partial\Omega,
\]
that is~\ref{hp3} is satisfied. Moreover, it is easy to see that $u_\eps$ admits at least one critical point in each $B_\delta(Q_i)$ for $i=1,2$.

Now let $Q\in B_\delta(Q_1)\setminus\{\orig\}$ be a critical point for $u$ such that $u(Q)\ge\frac{\sqrt e}{2}$. Then one has
\begin{align*}
f_\eps(Q&)^2\Delta\log f_\eps(Q)\\&={f_\eps(Q)\Delta f_\eps(Q)-\abs{\nabla f_\eps(Q)}^2}\\
&=-p\abs{Q}^{6\alpha}\abs{u(Q)}^{3p-1}+4\eps\alpha^2\abs{Q}^{2(\alpha-1)}\abs{u(Q)}^p-\eps p\abs{Q}^{4\alpha}\abs{u(Q)}^{2p-1}<0,
\end{align*}
for $\eps$ small enough and then~\ref{hp2} is not satisfied.
\end{rmk}
\newpage

\section{The nonlinear problem}
\label{sez3}

In this section we prove Theorem~\ref{thm2}. Hence, let $\Omega$ be a fixed bounded and convex domain in $\R^2$ and assume $B_1(\orig)\subseteq\Omega$. Let $\Omega_\eps$ be the family of smooth domains converging to $\Omega$ as $\eps\to 0$ and $u_\eps$ be the solution of problem~\ref{CAP4:PB1}, as in Theorem~\ref{thm2}.

\begin{proof}[Proof of Theorem~\ref{thm2}]
Since $\Omega$ is convex we know that if $u$ is a semi-stable solution of problem~\ref{CAP4:PB0} in $\Omega$, then it admits a unique nondegenerate critical point, which we denote by $P_c$.
Hence, it is enough to show that for all multi-indices $\alpha$ with $\abs{\alpha}\le2$ it holds
\begin{equation}
\label{CAP4:convergenza}
\sup_{\overline{\Omega_\eps\cap\Omega}}\left| D^\alpha\big(u_\eps-u\big)\right|\to 0,\quad\text{for }\eps\to0.
\end{equation}
Indeed, if $P_\eps\in\Omega_\eps$ is a critical point for $u_\eps$, then $P_\eps\to P_c$ as $\eps\to 0$, and it is a nondenerate maximum thanks to~\ref{CAP4:convergenza}. Then uniqueness follows from the convergence to $P_c$ and the nondegeneracy.

We prove~\ref{CAP4:convergenza} through several steps.\\

\emph{Step 1: there exists $C>0$ such that $\norm{u_\eps}_{H^2(\Omega_\eps)}\le C$.}\\
From classical regularity theory one has
\[
\norm{u_\eps}_{H^{m+2}(\Omega_\eps)}\le C(\Omega_\eps)\left(\norm{f(u_\eps)}_{H^m(\Omega_\eps)}+\norm{u_\eps}_{L^2(\Omega_\eps)}\right),
\]
but from the convergence of $\Omega_\eps$ to $\Omega$ one can see that $C(\Omega_\eps)$ does not really depend on $\eps$. Then for $m=0$ using the assumption $\norm{u_\eps}_{L^\infty(\Omega_\eps)}\le C$ we get the desired claim.\\

\emph{Step 2: $u_\eps\wto u$ in $H^1(\Omega^\rho )$, where $u$ is a semi-stable solution of problem~\ref{CAP4:PB0} in $\Omega$ and $\Omega^\rho:=\set{(x,y)\in\R^2| (\frac x\rho,\frac y\rho)\in\Omega}$, for some $\rho>1$.}\\
First of all from the convergence of $\Omega_\eps$ to $\Omega$ we can find $\rho>1$ such that $\Omega_\eps\subseteq \Omega^\rho$.
Since $\Omega_\eps$ are smooth we can consider $u_\eps$ defined in $\Omega^\rho$ by means of zero extension outside $\Omega_\eps$ and with a little abuse of notation we still denote such an extension by $u_\eps$.
Then from the previous step we have $u_\eps\wto u$ in $H^1(\Omega^\rho )$. Then it is easy to see that by means of the dominated convergence theorem it holds
\[
\int_\Omega\nabla u\nabla \phi=\int_\Omega f(u)\phi,\quad\text{for all }\phi\in\mathcal C^\infty_0(\Omega),
\]
and $u=0$ on $\partial\Omega$ in trace sense. Finally for any $\xi\in\mathcal C^\infty_0(\Omega)$ again the dominated convergence theorem gives
\[
\int_\Omega\abs{\nabla\xi}^2-\int_\Omega f'(u)\xi^2=\lim_{\eps\to0}\left(\int_\Omega\abs{\nabla\xi}^2-\int_\Omega f'(u_\eps)\xi^2\right)\ge0.
\]
Hence we proved that $u$ is a semi-stable solution of problem~\ref{CAP4:PB0} in $\Omega$.\\

\emph{Step 3: end of the proof}\\
From the convergence of $\Omega_\eps$ to $\Omega$ we can find $r>0$ and $Q_1,\dots,Q_k\in\partial\Omega$ such that
\begin{gather*}
\partial\Omega\subseteq\bigcup_{i=1}^kB_r(Q_i),\quad\text{and}\quad\partial\Omega_\eps\subseteq\bigcup_{i=1}^kB_r(Q_i),\\
\Omega\cap B_{2r}(Q_i)=\set{(x,y)\in B_{2r}(Q_i)|y>\Gamma^i(x)},\\
\Omega_\eps\cap B_{2r}(Q_i)=\set{(x,y)\in B_{2r}(Q_i)|y>\Gamma^i(x)+\gamma_\eps^i(x)},
\end{gather*}
where the last two relations hold up to a rotation and where $\Gamma^1,\dots,\Gamma^k,\gamma_\eps^1,\dots,\gamma_\eps^k$ are smooth functions such that
\[
\gamma_\eps^i\to0\quad\text{in }\mathcal C^4,\quad\text{for all }i=1,\dots,k.
\]

Now, let us fix $i=1$: up to a translation we can assume $Q_1=\orig$ and consider
\[
\bar u_\eps(x,y):=u_\eps(x,y+\gamma_\eps(x)),\quad \text{for all }(x,y)\in\Omega\cap B_{2r},
\]
where we omitted the index $i$ and $B_{2r}:=B_{2r}(\orig)$.
Then one has $\bar u_\eps=0$ on $\partial\Omega\cap B_{2r}$ and since $\gamma_\eps\to0$ in $\mathcal C^2$ it holds
\[
\sup_{\overline{\Omega_\eps\cap\Omega\cap B_{2r}}}\left| D^\alpha\big(u_\eps-\bar u_\eps\big)\right|\to 0,\quad\text{for }\eps\to0,
\]
for all multi-indices $\alpha$, with $\abs{\alpha}\le2$. Moreover, we have
\[
\begin{cases}
-\Delta(u-\bar u_\eps)=h_\eps&\text{in }\Omega\cap B_{2r}\\
u-\bar u_\eps=0&\text{on }\partial\Omega\cap B_{2r},
\end{cases}
\]
where
\[
h_\eps:=f(u)-f(\bar u_\eps)-2\partial_{xy}u_{\eps|(x,y+\gamma_\eps(x))}\dot\gamma_\eps-\partial_{yy}u_{\eps|(x,y+\gamma_\eps(x))}\dot\gamma_\eps^2-\partial_{y}u_{\eps|(x,y+\gamma_\eps(x))}\ddot\gamma_\eps.
\]
For $m\in\N$, by means of the mean value theorem and taking into account that $u$, $u_\eps$ and in turn $\bar u_\eps$ are uniformly bounded, we have $\abs{f(u)-f(\bar u_\eps)}\le C \abs{u-\bar u_\eps}$, then
\[
\norm{h_\eps}_{H^m(\Omega\cap B_{2r})}\le C\left(\norm{u-\bar u_\eps}_{H^{m}(\Omega\cap B_{2r})}+\norm{u_\eps}_{H^{m+2}(\Omega_\eps)}\norm{\gamma_\eps}_{\mathcal C^{m+2}}\right).
\]
Iterating the argument in \emph{Step 1}, we can find $C>0$ such that $\norm{u_\eps}_{H^{m+2}(\Omega_\eps)}\le C$ and
\[
\norm{u-\bar u_\eps}_{L^2(\Omega\cap B_{2r})}\le\norm{u- u_\eps}_{L^2(\Omega\cap B_{2r})}+C\norm{\bar u_\eps-u_\eps}_{L^\infty(\Omega\cap B_{2r})}\to0,\quad\text{for }\eps\to0,
\]
thanks to the compact embedding of $H^1(\Omega^\rho)$ in $L^2(\Omega^\rho)$. Then classical boundary regularity theory gives $\norm{u-\bar u_\eps}_{\mathcal C^2(\Omega\cap B_{r})}\to0$ for $\eps\to0$ and in turn
\[
\sup_{\overline{\Omega_\eps\cap\Omega\cap B_{r}}}\left| D^\alpha\big(u_\eps-u\big)\right|\to 0,\quad\text{for }\eps\to0,
\]
for all multi-indices $\alpha$, with $\abs{\alpha}\le2$. To complete the proof of~\ref{CAP4:convergenza} it is enough to repeat the argument for all $i=1,\dots,k$ and use interior regularity estimates taking into account that $\abs{-\Delta(u-u_\eps)}\le C\abs{u-u_\eps}$, for some $C>0$.
\end{proof}

\begin{rmk}
\label{rmk:hp:limitatezzanorma}
Here we show that if we assume that the nonlinearity has the form $f(u)=\lambda g(u)$, $g$ is smooth and satisfies~\ref{i2} and~\ref{i3}, that are
\begin{gather*}
g:\R\to\R\,\hbox{ is increasing and convex},\\
g(0)>0,
\end{gather*}
and $\lambda\in(0,\lambda^*(\Omega))$, then there exists $C>0$ such that $\norm{u_\eps}_{L^\infty(\Omega_\eps)}\le C$.

First of all note that, since $\Omega_\eps\to\Omega$ then $\lambda^*(\Omega_\eps)\to\lambda^*(\Omega)$ and then $\lambda<\lambda^*(\Omega_\eps)$ for $\eps$ small enough.

Remember that from the convergence of $\Omega_\eps$ to $\Omega$ we can find $\rho>1$ such that $\Omega_\eps\subseteq \Omega^\rho=\set{(x,y)\in\R^2| (x/\rho,y/\rho)\in\Omega}$. Moreover under this set of assumptions one has $\lambda^*(\Omega^\rho )>\lambda$ for $\eps$ small enough. Hence if we consider the stable solution $u^\rho $ of~\ref{CAP4:PB0} in $\Omega^\rho $ - using the convexity of $g$ - we have
\[
\begin{cases}
-\Delta(u_\eps-u^\rho )=\lambda \left(g(u_\eps)-g(u^\rho )\right)\le \lambda g'(u_\eps)(u_\eps-u^\rho )& \hbox{in}\ \Omega_\eps\\
u_\eps-u^\rho \le0& \hbox{on}\ \partial\Omega_\eps
\end{cases}
\]
and then from the stability of $u_\eps$ we can apply the maximum principle to deduce $u_\eps\le u^\rho \le\norm{u^\rho }_{L^\infty(\Omega^\rho )}$ in $\Omega_\eps$.\\
\end{rmk}
\newpage

\bibliographystyle{abbrv}
\bibliography{DCDSluca.bib}

\end{document}